\documentclass[a4paper,12pt]{article}
\usepackage{amsmath}
\usepackage{amssymb}
\usepackage{amsthm}
\usepackage{latexsym}

\def\opn#1#2{\def#1{\operatorname{#2}}} 
\opn\rank{rank} \opn\grade{grade} \opn\depth{depth} \opn\pd{pd}
\opn\height{ht} \opn\deg{deg} \opn\reg{reg} \opn\init{in}
\opn\Spec{Spec} \opn\coker{coker}

\date{}
\newtheorem{definition}{Definition}[section]
\newtheorem{theorem}{Theorem}[section]
\newtheorem{remark}{Remark}[section]
\newtheorem{example}{Example}[section]
\newtheorem{lemma}{Lemma}[section]
\newtheorem{corollary}{Corollary}[section]
\newtheorem{proposition}{Proposition}[section]
\newtheorem{property}{Property}[section]

\numberwithin{equation}{section}

\begin{document}
\title{\Large{\textbf{Monomial ideals of graphs with loops}}}

\author{\small{MAURIZIO \,IMBESI}\\[-1mm]
\small{\em Department of Mathematics, University of Messina}\\[-2mm]
\small{\em Viale F. Stagno d'Alcontres, 31 $-$ I 98166 Messina, Italy}\\[-2mm]
\small{{\em e-mail:} imbesim@unime.it}\\[2mm]
\small{MONICA \,LA BARBIERA}\\[-1mm]
\small{\em Department of Mathematics, University of Messina}\\[-2mm]
\small{\em Viale F. Stagno d'Alcontres, 31 $-$ I 98166 Messina, Italy}\\[-2mm]
\small{{\em e-mail:} monicalb@unime.it}} \maketitle

\thispagestyle{empty} \vspace{-5mm}
\par\noindent
\begin{center} {\small \bf Abstract} \end{center}
\hspace{4mm}{\footnotesize We investigate, using the notion of linear quotients, significative classes of connected graphs whose monomial edge ideals, not necessarily squarefree, 
have linear resolution,
in order to compute standard algebraic invariants of the polynomial ring related to these graphs modulo such ideals. Moreover we are able to determine the structure of the ideals of vertex covers for the edge ideals associated to the previous classes of graphs which can have loops on any vertex. Lastly, it is shown that these ideals are of linear type.
\vspace{5mm} \par\noindent
Ideals of linear type.

\vspace{1mm} \par\noindent}
\section*{Introduction}

\hspace{5mm} In \cite{IL1} it was first shown how some algebraic properties related to remarkable classes of graphs hold when appropriate loops are put.\par
In this article we are interested in studying standard algebraic properties of monomial ideals arising from the edges of some graphs, the so-called edge ideals, and exposing situations for which such properties are preserved when we add loops to them.
The generators of ideals of vertex covers for the edge ideals associated to graphs with loops are also examined, proving that these ideals are of linear type.\\
Let $\mathcal{G}$ be a graph on $n$ vertices $v_1,\ldots,v_n$\,. An algebraic object attached to $\mathcal{G}$ is the edge ideal $I(\mathcal{G})$, a monomial ideal of the polynomial ring in $n$ variables $R = K[X_1,\dots, X_n]$, $K$ a field.\\
When $\mathcal{G}$ is a loopless graph, $I(\mathcal{G})$ is generated by squarefree monomials of degree two in $R,\;
I(\mathcal{G})=\big(\{X_iX_j \;|\; \{v_i,v_j \}$ is an edge of $\mathcal{G}\,\}\big)$, but if $\mathcal{G}$ is a graph having loops
$\{v_i,v_i \}$, among the generators of $I(\mathcal{G})$ there are also non-squarefree monomials $X_i^2, \, i=1,\ldots,n$.\\
For a relevant class of connected graphs with loops we prove that their edge ideals have linear resolution, by using the technique of studying the linear quotients of such ideals, as previously employed in \cite{L}.\\
More precisely, it is examined a wide class of squarefree edge
ideals associated to the connected graphs $\mathcal{H}$, on $n$
vertices, consisting of the union of a complete graph $K_m$, $m<n$,
and star graphs with centers the vertices of $K_m$.
By adding loops on some vertices of $K_m$, we introduce a new class of
non-squarefree edge ideals associated to the connected graphs\,
$\mathcal{K}=\mathcal{H} \cup \{v_j,v_j\}$, for some $j=1,\ldots,m$\,.
We prove that $I(\mathcal{H})$ and $I(\mathcal{K})$ have
linear resolution. We also give formulae
for standard invariants of $R/I(\mathcal{H})$ and
$R/I(\mathcal{K})$ such as dimension, projective dimension, depth, Castelnuovo-Mumford regularity.\\
In \cite{IL} the computation of such invariants is made for the symmetric algebras of subclasses of squarefree
edge ideals generated by $s$-sequences associated \mbox{to $\mathcal{G}$.}\\
Some algebraic aspects linked to the minimal vertex covers for such
classes of graphs can be considered. Keeping in mind the one to one
correspondence between minimal vertex covers of any graph and
minimal prime ideals of its edge ideal, we generalize to a graph
with loops the notion of ideal of (minimal) vertex covers
and determine the structure of the ideals of vertex covers
$I_c(\mathcal{H})$ and $I_c(\mathcal{K}')$ for the classes of edge
ideals associated to $\mathcal{H}$ and $\mathcal{K}'=\mathcal{H}
\cup \{v_i,v_i\}$, for some $i=1,\ldots,n$\,. We may observe that $\mathcal{K}'$ is larger than $\mathcal{K}$ because
loops on $\mathcal{K}'$ stay also on vertices that don't belong to
$K_m$\,. Moreover we prove that the symmetric and the Rees algebras
of $I_c(\mathcal{H})$ and $I_c(\mathcal{K}')$
over $R$ are isomorphic, namely such ideals of linear type.\\
The work is subdivided as follows. In section 1 the classes
of graphs $\mathcal{H}, \mathcal{K}$ and their edge ideals are analyzed. In particular, starting from $\mathcal{H}$, we introduce $\mathcal{K}$ and consider its edge
ideal whose ordered generators are
$X_1X_{m+1},X_1X_{m+2},...,$ $X_1X_{m+i_1}, X_1X_2,
X_2X_{m+i_1+1}, \ldots, X_{2}X_{m+i_1+i_2}, X_2X_{3},X_1X_{3},
X_3X_{m+i_1+i_2+1},...,$ $X_{3}X_{m+i_1+i_2+i_3},
X_3X_{4},X_2X_{4},X_1X_4,..., X_mX_{m+i_1+\cdots +
i_{m-1}+1},...,X_{m}X_{m+i_1+\cdots +i_m},$
$X_{t_1}^2, \ldots, X_{t_l}^2$, where $\{t_1,\ldots,t_l\}$ $\subseteq
\{1, \ldots, m\}$ and $n=m+i_1+\cdots +i_m$\,.\\ According to results
that characterize monomial ideals with linear quotients
(\cite{{CH},{HH2}}), it is showed that the edge ideals of
$\mathcal{H}$ and $\mathcal{K}$ have linear resolution.
As an application, standard algebraic invariants are calculated.\\
In section 2 we examine the structure of the ideals of vertex covers
for the classes of edge ideals associated to $\mathcal{H}$ and
$\mathcal{K}'$ using the description of the ideals of vertex
covers for the edge ideals associated to the complete graphs and
the star graphs that make them up.\\
We remark that the ideal of vertex covers of $I(\mathcal{K}')$,
 when $\mathcal{K}'$ has loops on all its $n$ vertices, has the unique generator \,$X_1\cdots X_{m} \cdots X_n$\,.\\
In section 3 we investigate the Rees algebra of $I_c(\mathcal{H})$
and $I_c(\mathcal{K}')$. We recall that if $I=(f_1,\ldots,f_s)$ is
an ideal of $R$, the Rees algebra $\Re(I)$ of $I$ is defined to be
the $R$-graded algebra $\bigoplus_{i\geqslant 0}I^i$.  Let
$\varphi:R[T_1,\ldots,T_s] \rightarrow \Re(I)=R[f_1t,\ldots,f_st]$
be an epimorphism of graded $R$-algebras defined by
$\varphi(T_i)=f_it$\,, $i=1,\ldots,s$\,, and $N=\ker\varphi$ be the
ideal of presentation of $\Re(I)$. If $N$ is generated by linear
relations, namely $R$-homogeneous elements of degree $1$, then $I$
is said of linear type. Several classes of ideals of $R$ of linear
type are known. For instance, ideals generated by $d$-sequences and
$M$-sequences are of linear type (\cite{{H},{V},{CD}}). We show that
the ideals of vertex covers $I_c(\mathcal{H})$ and
$I_c(\mathcal{K}')$ are of linear type.


\section{Linear resolutions and invariants}

\hspace{6mm} Let $R = K[X_1,\ldots,X_n]$ be the polynomial ring in $n$ variables
over a field $K$ with $\deg X_i=1$, for all $i=1,\ldots,n$. For a
monomial ideal $I \subset R$ we denote by $G(I)$ its unique minimal
set of monomial generators.

\begin{definition}\rm{ A \textit{vertex cover} of a monomial ideal $I \subset R$
is a subset $C$ of $\{ X_1,\ldots,X_n \}$ such that each $u\in G(I)$
is divided by some $X_i \in C$.\\ The vertex cover $C$ is called
\textit{minimal} if no proper subset of $C$ is a vertex cover of $I$.}
\end{definition}

\noindent Let $h(I)$ denote the minimal cardinality of the vertex covers of $I$.

\begin{definition}\rm{
A monomial ideal $I \subset R$ is said to have \textit{linear quotients} if there is an
ordering $u_1, \ldots, u_t$ of monomials belonging to $G(I)$ with
$\deg (u_1) \leqslant \cdots \leqslant \deg (u_t)$ such that
the colon ideal $(u_1, \ldots, u_{j-1}):(u_j)$
is generated by a subset of $\{X_1,\ldots,X_n \}$, for $2\leqslant j \leqslant t$\,.}
\end{definition}

\begin{remark}\rm{
A monomial ideal $I$ of $R$ generated in one degree that has
linear quotients admits a linear resolution (\cite{CH}, Lemma 4.1).}
\end{remark}

\noindent For a  monomial ideal $I$ of $R$ having linear quotients with
respect to the ordering $u_1, \ldots, u_{t}$ of the monomials of
$G(I)$, let $q_j(I)$ denote the number of the variables which is
required to generate the ideal $(u_1, \ldots, u_{j-1}):(u_j)$\,, and
set $q(I)= \textrm{max}_{2\leqslant j\leqslant t}\, q_j(I)$.

\begin{remark}\rm{
The integer $q(I)$ is independent on the choice of the ordering of
the generators that gives linear quotients (\cite{HT}).}
\end{remark}

\noindent In this section we study classes of monomial ideals generated
in degree two arising from graphs that have linear quotients, that is
equivalent to say that they have linear resolution (\cite{HH2}).\\[1mm]
Let's introduce some preliminary notions.\\
Let $\mathcal{G}$ be a graph and  $V(\mathcal{G})= \{v_1, \ldots,
v_n\}$ be the set of its vertices. We put $E(\mathcal{G})
=\{\{v_i,v_j\} | v_i\neq v_j, \ \ v_i, v_j \in V(\mathcal{G}) \} $
the set of edges of $\mathcal{G}$ and $L(\mathcal{G}) =\{\{v_i,v_i\}
|  \ \ v_i \in V(\mathcal{G}) \} $ the set of loops of
$\mathcal{G}$. Hence $\{v_i,v_j\}$ is an edge joining $v_i$ to $v_j$
and $\{ v_i,v_i \}$ is a loop of the vertex $v_i$. Set
$W(\mathcal{G}) = E(\mathcal{G}) \cup
L(\mathcal{G})$.\\
If $L(\mathcal{G}) = \emptyset$, the graph $\mathcal{G}$ is said
simple or loopless, otherwise
$\mathcal{G}$ is a graph with loops.\\
A graph $\mathcal{G}$ on $n$ vertices $v_1,\ldots, v_n$ is
\textit{complete} if there exists an edge for all pairs $\{v_i,v_j
\}$ of vertices of $\mathcal{G}$. It is denoted by $K_n$.\\[1mm]
If $V(\mathcal{G})=\{v_1,\ldots,v_n\}$ and $R=K[X_1,\ldots,X_n]$ is
the polynomial ring over a field $K$ such that each variable $X_i$
corresponds to the vertex $v_i$, the \textit{edge ideal}
$I(\mathcal{G})$ associated to $\mathcal{G}$ is the ideal
$\big(\{X_{i}X_{j}\,|\,\{v_{i},v_{j}\}\in W(\mathcal{G})\}\big) \subset R$\,.
\par\noindent
Note that the non-zero edge ideals are those generated by monomials of degree 2. This implies that $I(\mathcal{G})$ is a graded ideal of $S$ of initial degree 2, that
is $I(\mathcal{G})=\oplus_{i\geqslant 2} \,(I(\mathcal{G})_{i})$\,.
If $W(\mathcal{G})= \emptyset$, then $I (\mathcal{G})=(0)$\,.
\vspace{3mm}
\par
First it is examined a relevant wide class of squarefree edge ideals
associated to connected graphs $\mathcal{H}$, on $n$ vertices,
consisting of the union of a complete graph $K_m$, $m<n$, and star
graphs in the vertices of
$K_m$.\\
More precisely,  $\mathcal{H}=K_m \cup $\;star$_j(k)$,  where $K_m$
is the complete graph on $m$ vertices $v_1, \ldots, v_m$, $m\!<\!n,$
and \,star$_j(k)$ is the star graph on $k$ vertices with center
$v_j$,
for some $j=1,\ldots,m, \,  k\!\leq \!n\!-\!m$. One has:\\
$I(\mathcal{H})=$ $(X_1X_{m+1},X_1X_{m+2}, \ldots, X_{1}X_{m+i_1},
X_1X_{2}, X_2X_{m+i_1+1},X_2X_{m+i_1+2}, \ldots,$
$X_{2}X_{m+i_1+i_2}, X_2X_{3},X_1X_{3},
X_3X_{m+i_1+i_2+1},X_3X_{m+i_1+i_2+2}, \ldots,
X_{3}X_{m+i_1+i_2+i_3},$\\
$X_3X_{4},$ $X_2X_{4},X_1X_4, \ldots, X_mX_{m+i_1+i_2+\cdots +
i_{m-1}+1},X_mX_{m+i_1+i_2+\cdots +
i_{m-1}+2}, \ldots,$\\
$ X_{m}X_{m+i_1+i_2+\cdots + i_{m-1}+i_m})\subset R=K[X_1,\ldots,
X_n]$.\\ $|G(I(\mathcal{H}))|=i_1+i_2+\cdots +i_m+\frac{m(m-1)}{2}$.

\begin{proposition}\label{Prop1}
$I(\mathcal{H})$ has a linear resolution.
\end{proposition}
\begin{proof}
$I(\mathcal{H})$
is the edge ideal of the graph $\mathcal{H}$ with
$n=m+i_1+i_2+\cdots + i_m$ vertices and $\ell= i_1+i_2+\cdots +i_m +
\frac{m(m-1)}{2}=n-m + \frac{m(m-1)}{2}= n+  \frac{m(m-3)}{2}$
edges. Set $f_1, \ldots, f_{\ell}$\, the squarefree monomial
generators of $I(\mathcal{H})$. It results:\\
$(f_1):(f_2)=(X_{m+1})$, \\
$(f_1,f_2):(f_3)=(X_{m+1},X_{m+2})$,\\
$...................................$,\\
$(f_1, \ldots, f_{i_1-1}):(f_{i_1})=(X_{m+1},X_{m+2}, \ldots, X_{m+i_1-1})$,\\
$(f_1, \ldots, f_{i_1}):(f_{i_1 +1})=(X_{m+1},X_{m+2}, \ldots, X_{m+i_1})$,\\
$(f_1, \ldots, f_{i_1 + 1}):(f_{i_1+2})=(X_1X_{m+1},X_1X_{m+2},
\ldots, X_1X_{m+i_1},X_1)=(X_1)$,\\
$(f_1, \ldots, f_{i_1+2}):(f_{i_1 +3})=(X_{1}, X_{m+i_1+1})$,\\
$...................................$,\\
$(f_1, \ldots,
f_{i_1+i_2+1}):(f_{i_1 +i_2+2})=(X_{1}, X_{m+i_1+1}, \ldots,
X_{m+i_1+i_2})$,\\ $(f_1, \ldots, f_{i_1+i_2+2}):(f_{i_1
+i_2+3})=(X_{m+1},X_{m+2}, \ldots, X_{m+i_1}, X_{2})$,\\ $(f_1,
\ldots, f_{i_1+i_2+3}):(f_{i_1 +i_2+4})=(X_{2}, X_{1})$,\\
$...................................$,\\
$(f_1, \ldots, f_{i_1+i_2+i_3+2}):(f_{i_1 +i_2+i_3+3})=(X_{2},X_{1},
X_{m+i_1+i_2+1}, \ldots, X_{m+i_1+i_2+i_3})$,\\ and so on up to
\\$(f_1, \ldots, f_{\ell-1}):(f_{\ell})=(X_{1}, \ldots, X_{m-1},
X_{m+i_1+i_2+\cdots +i_{m-1}+1}, \ldots,
X_{m+i_1+i_2+\cdots +i_{m}-1})$.\\[1mm]
Hence $I(\mathcal{H})$ has linear quotients. According to results in
\cite{HH2} about monomial ideals with linear quotients, it follows
that the edge ideal $I(\mathcal{H})$ has a linear resolution.
\end{proof}
\begin{remark}\rm{
R. Fr\"oberg proved that the edge ideal of a simple graph has a
linear resolution if and only if its complementary graph is chordal
(\cite{HH2}, Theorem 9.2.3). By using such a characterization, it is possible to show that the edge
ideal $I(\mathcal{H})$ of Proposition \ref{Prop1} has a linear resolution.}
\end{remark}

\noindent The study of the linear quotients is useful in order to investigate algebraic invariants
of $R/I(\mathcal{H})$: the \emph{dimension}, $\dim_R(R/I(\mathcal{H}))$,
the \emph{depth}, \,depth$(R/I(\mathcal{H}))$, the \emph{projective dimension},
${\pd}_R(R/I(\mathcal{H}))$ and the \emph{Castelnuovo-Mumford regularity},
$\reg_R(R/I(\mathcal{H}))$.

\begin{lemma}\label{1.1}
Let $R=K[X_1,\ldots,X_n]$ and $I(\mathcal{H}) \subset R$. Then:
\vspace{-1mm}
$$ h(I(\mathcal{H})) = m\,.$$
\end{lemma}
\begin{proof}  The minimal cardinality of the vertex
covers of $I(\mathcal{H})$ is $h(I(\mathcal{H}))= m$, being
$C=\{X_1,\ldots,X_m \}$ a minimal vertex cover of $I(\mathcal{H})$
by construction. \end{proof}

\begin{lemma}\label{1.2}
Let $R=K[X_1,\ldots,X_n]$  and $I(\mathcal{H}) \subset R$. Then:
\vspace{-1mm}
$$q(I(\mathcal{H})) = m + \textrm{max}_{1\leqslant j\leqslant m}\, i_j -2\,.$$
\end{lemma}
\begin{proof}
By the computation of the linear quotients (Proposition
\ref{Prop1}), the maximum number of the variables which is required
to generate the ideal $(f_1, \ldots, f_{h-1}):(f_h)$, for $h=1,
\ldots, \ell$, is given by $(m-1)+\textrm{max}_{1\leqslant
j\leqslant m}\, i_j -1$. It follows that $q(I(\mathcal{H}))= m +
\textrm{max}_{1\leqslant j\leqslant m}\, i_j -2\,.$
\end{proof}
\begin{theorem}
Let $R=K[X_1,\ldots,X_n]$  and $I(\mathcal{H}) \subset R$. Then:
\item{1)} $\textrm{dim}_R(R/I(\mathcal{H})) = n-m$\,.
\item{2)} $\textrm{pd}_R(R/I(\mathcal{H}))= m +\textrm{max}_{1\leqslant j\leqslant m}\,
i_j -1$\,.
\item{3)} $\textrm{depth}_R(R/I(\mathcal{H}))=n-m - \textrm{max}_{1\leqslant j\leqslant m}\,
i_j +1$\,.
\item{4)} $\textrm{reg}_R(R/I(\mathcal{H}))= 1$\,.
\end{theorem}
\begin{proof}  1) One has $\textrm{dim}_R(R/I(\mathcal{H})) =
\textrm{dim}_R R - h(I(\mathcal{H}))$ (\cite{HH1}). Hence, by Lemma \ref{1.1}\,,
$\textrm{dim}_R(R/I(\mathcal{H})) =n-m$\,.\\
2) The length of the minimal free resolution of $R/I(\mathcal{H})$
over $R$ is equal to $q(I(\mathcal{H})) + 1$ (\cite{HT}, Corollary
1.6). Then $\textrm{pd}_R(R/I(\mathcal{H}))= m +
\textrm{max}_{1\leqslant j\leqslant m}\,
i_j -1$.\\
3) As a consequence of 2), by Auslander-Buchsbaum
formula, one has \\
$\textrm{depth}_R(R/I(\mathcal{H}))=n
-\textrm{pd}_R(R/I(\mathcal{H}))= n-m -\textrm{max}_{1\leqslant
j\leqslant m}\,
i_j +1 $.\\
4) $I(\mathcal{H})$ has a linear resolution, then
$\textrm{reg}_R(R/I(\mathcal{H}))=1$.
\end{proof}


\vspace{1mm}
\par
Starting from the class of edge ideals associated to $\mathcal{H}$ and adding loops on some vertices of $K_m$, we now
analyze a larger class of non-squarefree edge ideals associated to
connected graphs\, $\mathcal{K}=\mathcal{H} \cup \{v_j,v_j\}$, for
some $j=1,\ldots,m$\,.\\
 Let $\mathcal{K}$ be the connected graph
with $n$ vertices $v_1, \ldots, v_n$ and $K_m$, $m<n$, be the
complete subgraph of $\mathcal{K}$ with vertices $v_1, \ldots, v_m$,
such that \\$I(\mathcal{K})= (X_1X_{m+1},X_1X_{m+2}, \ldots,
X_{1}X_{m+i_1}, X_1X_{2}, X_2X_{m+i_1+1},X_2X_{m+i_1+2}, \ldots,$ $
X_{2}X_{m+i_1+i_2}, X_2X_{3},X_1X_{3},
X_3X_{m+i_1+i_2+1},X_3X_{m+i_1+i_2+2}, \ldots,
X_{3}X_{m+i_1+i_2+i_3},$\\$ X_3X_{4},X_2X_{4},X_1X_4, \ldots,
X_mX_{m+i_1+i_2+\cdots + i_{m-1}+1},X_mX_{m+i_1+i_2+\cdots +
i_{m-1}+2}, \ldots,$\\$ X_{m}X_{m+i_1+i_2+\cdots + i_{m-1}+i_m},
X_{t_1}^2, \ldots, X_{t_l}^2)\subset R=K[X_1,\ldots, X_n]$, with
$\{t_1,\ldots,t_l\}$ $\subseteq \{1, \ldots, m\}$ and
$n=m+i_1+\cdots
+i_m$.\\
$|G(I(\mathcal{K}))|=n-(m-l)+\frac{m(m-1)}{2}$\,.\\
Taking in account the notion of linear quotients, it is proved that the edge ideal $I(\mathcal{K})$ has
still a linear resolution.

\begin{proposition}\label{P2}
$I(\mathcal{K})$ has a linear resolution.
\end{proposition}
\begin{proof}
$I(\mathcal{K})$ is the edge ideal of the graph $\mathcal{K}$ with
$n=m+i_1+i_2+\cdots + i_m$ vertices, $\ell= i_1+i_2+\cdots +i_m +
\frac{m(m-1)}{2}=n-m + \frac{m(m-1)}{2}= n+ \frac{m(m-3)}{2}$ edges
and $l \leqslant m$ loops
$\{v_{t_1},v_{t_1}\},\ldots,\{v_{t_l},v_{t_l}\}$. Set $f_1, \ldots,
f_{\ell}, f_{\ell+1}, \ldots, f_{\ell+l}$ the monomial generators of
$I(\mathcal{K})$, of which the last $l$ are not squarefree.\\
The ideals $(f_1, \ldots, f_{h-1}):(f_h)$, for $h=1, \ldots, \ell$, have been computed in Proposition \ref{Prop1}.
 Moreover, it results: \\
$(f_1, \ldots, f_{\ell}):(f_{\ell+1}) = (X_2, \ldots,X_m, X_{m+1},
\ldots, X_{m+i_1})$, if $v_1$
has a loop, \\
$(f_1, \ldots, f_{\ell}):(f_{\ell+1}) = (X_1, \ldots,X_{t_1-1}, X_{t_1+1}, \ldots, X_m, X_{m+i_1+\ldots+i_{t_1-1}+1},\ldots,$\\$X_{m+i_1+\ldots +i_{t_1-1}+i_{t_1}})$, if $\{v_{t_1},v_{t_1}\}\neq \{v_1,v_1\}$;\\[1mm]
$...................................$,\\[.8mm]
$(f_1, \ldots, f_{\ell+l-1}):(f_{\ell+l}) = (X_1, \ldots, X_{t_l-1},
X_{t_l+1}, \ldots, X_m, X_{m+i_1+\ldots+i_{t_l-1}+1},\ldots,$\\$ X_{m+i_1+\ldots+i_{t_l-1}+i_{t_l}})$, if $\{v_{t_l},v_{t_l}\}\neq \{v_m,v_m\}$,\\
$(f_1, \ldots, f_{\ell+l-1}):(f_{\ell+l}) = (X_1, \ldots,X_{m-1},
_{m+i_1+\ldots+i_{m-1}+1}, \ldots, X_{m+i_1+\ldots + i_m})$, if $v_m$ has a loop. \\[1mm]
Hence $I(\mathcal{K})$ has linear quotients, that is equivalent to
say that $I(\mathcal{K})$ has a linear resolution (\cite{HH2}).
\end{proof}
\noindent The study of the linear quotients as in Proposition
\ref{P2} is useful in order to investigate algebraic invariants of
$R/I(\mathcal{K})$.
\begin{lemma}\label{1.3}
Let $R=K[X_1,\ldots,X_n]$ and $I(\mathcal{K})\subset R$. Then:
\vspace{-1mm}
$$ h(I(\mathcal{K}))= m\,.$$
\end{lemma}
\begin{proof}  The minimal cardinality of the vertex
covers of $I(\mathcal{K})$ is $h(I(\mathcal{K}))= m$, being
$C=\{X_1,\ldots,X_m \}$ a minimal vertex cover of $I(\mathcal{K})$
by construction. \end{proof}

\begin{lemma}\label{1.4}
Let $R=K[X_1,\ldots,X_n]$ and $I(\mathcal{K})\subset R$. Then:
\vspace{-1mm}
$$q(I(\mathcal{K})) = m +\textrm{max}_{\,t_1\leqslant j\leqslant \,t_{\mu}}\, i_j -1, \ \ \mu \leqslant m\,.$$
\end{lemma}
\begin{proof}
By the computation of the linear quotients (Proposition \ref{P2}),
the maximum number of the variables which is required to generate
the ideal $(f_1, \ldots, f_{k-1}):(f_k)$, for $k=1, \ldots, \ell+\mu$,
$\mu \leqslant m$, is given by $(m-1)+\textrm{max}_{\,t_1\leqslant
j\leqslant \,t_{\mu}}\, i_j$. It follows that $q(I(\mathcal{K}))= m
+\textrm{max}_{\,t_1\leqslant j\leqslant \,t_{\mu}}\, i_j -1$\,.
\end{proof}

\begin{theorem}\label{T4}
Let $R=K[X_1,\ldots,X_n]$  and $I(\mathcal{K}) \subset R$. Then:
\item{1)} $\textrm{dim}_R(R/I(\mathcal{K})) = n-m$\,.
\item{2)} $\textrm{pd}_R(R/I(\mathcal{K}))= m +\textrm{max}_{\,t_1\leqslant
j\leqslant \,t_{\mu}}\, i_j$, $\mu \leqslant m$\,.
\item{3)} $\textrm{depth}_R(R/I(\mathcal{K}))=n- m -\textrm{max}_{\,t_1\leqslant j\leqslant \,t_{\mu}}\,
i_j$, $\mu \leqslant m$\,.
\item{4)} $\textrm{reg}_R(R/I(\mathcal{K}))= 1$\,.
\end{theorem}
\begin{proof}  1) One has $\textrm{dim}_R(R/I(\mathcal{K})) =
\textrm{dim}_R R - h(I(\mathcal{K}))$ (\cite{HH1}). Hence, by Lemma \ref{1.3}\,,
$\textrm{dim}_R(R/I(\mathcal{K})) =n-m$\,.\\
2) The length of the minimal free resolution of $R/I(\mathcal{K}))$
over $R$ is equal to $q(I(\mathcal{K})) + 1$ (\cite{HT}, Corollary
1.6). Then $\textrm{pd}_R(R/I(\mathcal{K}))= m
+\textrm{max}_{\,t_1\leqslant
j\leqslant \,t_{\mu}}\, i_j$, $\mu \leqslant m$.\\
3) As a consequence of 2), by Auslander-Buchsbaum formula, one has \\
$\textrm{depth}_R(R/I(\mathcal{K}))=n
-\textrm{pd}_R(R/I(\mathcal{K}))= n- m -\textrm{max}_{\,t_1\leqslant
j\leqslant \,t_{\mu}}\,
i_j$, $\mu \leqslant m$.\\
4) $I(\mathcal{K})$ has a linear resolution, then
$\textrm{reg}_R(R/I(\mathcal{K}))=1$. \end{proof}
\begin{remark}\rm{
When the graph $\mathcal{K}$ has at least a loop on a vertex that
don't belong to $K_m$, then its edge ideal has not linear quotients.
In fact, it can be verified there is no ordering of the monomials
$f_1,\ldots,f_s \in G(I(\mathcal{K}))\,,\,
s=n-(m-l)+\frac{m(m-1)}{2}\,,\,$ such that
$(f_1,\ldots,f_{j-1}):(f_j)$ is generated by a subset of
$\{X_1,\ldots,X_n\}$, for $2 \leqslant j \leqslant s$. }
\end{remark}


\section{Ideals of vertex covers}

\begin{definition}\rm{ Let $\mathcal{G}$ be a graph with
vertex set $V(\mathcal{G})=\{v_1,\ldots, v_n\}$. A subset $C$ of $V(G)$ is said a \textit{minimal vertex cover} for $\mathcal{G}$ if:\\
(1) every edge of $\mathcal{G}$ is incident with one vertex in $C$; \\
(2) there is no proper subset of $C$ with this property.\\
If $C$ satisfies condition (1) only, then
$C$ is called a \textit{vertex cover} of $\mathcal{G}$ and
$C$ is said to cover all the edges of $\mathcal{G}$.}
\end{definition}
\noindent The smallest number of vertices in any  minimal vertex cover of $\mathcal{G}$
is said  \textit{vertex covering number}.  We denote it by $\alpha_0(\mathcal{G})$.\\[1.5mm]
We consider some algebraic aspects linked to the minimal vertex covers of a graph $\mathcal{G}$ with set of edges $E(\mathcal{G})$ and set of loops $L(\mathcal{G})$.\\[.5mm]
Let $I(\mathcal{G}) = \big(\{X_iX_j \ \ | \ \{v_i,v_j\} \in W(\mathcal{G})=E(\mathcal{G})\cup L(\mathcal{G})\}\big)$ be the edge ideal associated to $\mathcal{G}$. A loop $\{v_i,v_i\} \in L(\mathcal{G})$, such that $v_i$ belongs to a vertex cover of $\mathcal{G}$, can be thought to have a double covering that preserves the minimality.\\
There exists a one to one correspondence between the minimal vertex
covers of $\mathcal{G}$ and the minimal prime ideals of
$I(\mathcal{G})$.\\ In fact, $\wp$ is a minimal prime ideal of
$I(\mathcal{G})$ if and only if $\wp=(C)$, for some minimal vertex
cover $C$ of $\mathcal{G}$ (\cite{V1}, Prop. 6.1.16). Hence ht
$I(\mathcal{G})=\alpha_0(\mathcal{G})$.\\ Thus the primary
decomposition of the edge ideal of $\mathcal{G}$ is given by
$I(\mathcal{G})=(C_1) \cap \cdots \cap (C_p)$, where $C_1, \ldots,
C_p\,$ are the minimal vertex covers of $\mathcal{G}$.

\begin{definition}\rm{
Let $I \subset R = K[X_1, \ldots , X_n]$ be a monomial ideal. The \textit{ideal of (minimal) covers} of $I$, denoted by $I_c$\,, is the ideal of $R$ generated by all monomials $X_{i_1}\cdots X_{i_k}$ such that $(X_{i_1},\ldots, X_{i_k})$ is an associated (minimal) prime ideal of $I$.\\
If $I(\mathcal{G})$ is the edge ideal of a graph $\mathcal{G}$, we call $I_c(\mathcal{G})$ the \textit{ideal of vertex covers} of $I(\mathcal{G})$.
}
\end{definition}
\begin{property} \label{vertcov}
$I_c(\mathcal{G})=
\displaystyle{\Big(\!\bigcap_{\stackrel{\{v_i,v_{\!j}\}\in E(\mathcal{G})}{i \neq j}}\!\!\! (X_i,X_j)}\Big) \,\cap\, \big(X_k \ | \ \{v_k,v_k\}\in L(\mathcal{G}), k \neq i,j\big) $.
\end{property}
We want to examine the structure of the ideals of vertex covers for
the class of squarefree edge ideals associated to $\mathcal{H}$.  In
the sequel, we rename $v_{\alpha_1}, \ldots, v_{\alpha_m}$ the vertices of the complete graph $K_m$\,,
 $1 \leqslant \alpha_1 < \alpha_2 < \ldots < \alpha_m=n$, and $\alpha_1=1$
 if there is not the star graph with center $v_{\alpha_1}$ .\\
The following two lemmas give explicitly the generators of the
ideals of vertex covers for the edge ideals of a complete graph and
a star graph, respectively.\\ Their proofs are an easy application
of Property \ref{vertcov}\,.
\begin{lemma} \label{Lema1}
The ideal of vertex covers of $\,I(K_m)$ is generated by $m$ monomials and it is \, $I_c(K_m)=(X_2X_3\cdots X_m\,,X_1X_3\cdots X_m\,,\,\ldots, X_1X_2\cdots X_{m-1})\,.$
\end{lemma}
\begin{lemma} \label{Lema2}
The ideal of vertex covers of $\,I(star_j(k))$ is generated by two
monomials, namely $X_1\cdots X_{j-1}X_{j+1}\cdots X_{k}\,, \,
X_j$\,.
\end{lemma}
Let's study the structure of the ideal of vertex covers of $I(\mathcal{H})$ when at least a vertex of $K_m$ has degree $m-1$.
\begin{proposition} \label{Km+1star}
Let $\mathcal{H}$ be the connected graph with $n$ vertices $v_1,\ldots, v_n$ formed by the union of a complete graph
$K_m, m<n$, with vertices $v_{\alpha_1}, \ldots, v_{\alpha_m},$ $1 \leqslant \alpha_1 < \alpha_2 < \ldots < \alpha_m=n$,
and of a star graph star$_{\alpha_1}(\alpha_1)$ on vertices $v_1,\ldots, v_{\alpha_1}$,
or a star graph star$_{\alpha_i}(\alpha_i-\alpha_{i-1})$ with center $v_{\alpha_i}$, for some $i=2,\ldots,m$\,.
The ideal of vertex covers $I_c(\mathcal{H})$ has $m$ monomial squarefree generators and it is:\\
$(X_{1}\!\cdot\!\cdot\!\cdot X_{\alpha_1-1}X_{\alpha_2}\!\cdot\!\cdot\!\cdot X_{\alpha_m},
X_{\alpha_1}X_{\alpha_3}\!\cdot\!\cdot\!\cdot X_{\alpha_m},\ldots,X_{\alpha_1}\!\cdot\!\cdot\!\cdot X_{\alpha_{m-1}}),$ if \,$i=1$,\\
$(X_{\alpha_2}\!\cdot\!\cdot\!\cdot X_{\alpha_m},\ldots,
X_{\alpha_1}\!\cdot\!\cdot\!\cdot
X_{\alpha_{i-1}}X_{\alpha_{i-1}+1}\!\cdot\!\cdot\!\cdot
X_{\alpha_{i+1}-1}X_{\alpha_{i+1}}\!\cdot\!\cdot\!\cdot
X_{\alpha_m},\ldots,X_{\alpha_1}\!\cdot\!\cdot\!\cdot
X_{\alpha_{m-1}}),$ otherwise.
\end{proposition}
\begin{proof}
Because of the one to one correspondence between minimal vertex covers of a graph and minimal prime ideals of its edge ideal, for the complete graph $K_m$ it follows that $I(K_m)=\displaystyle{\bigcap_{i=1}^m \mathcal{P}_{\alpha_i}}\,,$ \,where $\mathcal{P}_{\alpha_i}=(X_{\alpha_1}\,, \ldots, X_{\alpha_{i-1}}\,,$ $X_{\alpha_{i+1}}\,,$ $\ldots,  X_{\alpha_m})$\,. The vertex covers of the graph $\mathcal{H}$ will be:\\[1mm]
$m-1$ vertex covers of $K_m$, but
$\{v_{\alpha_1}\,, \ldots, v_{\alpha_{i-1}}\,,v_{\alpha_{i+1}}\,, \ldots, v_{\alpha_m}\}$, \mbox{by Lemma \ref{Lema1};}\\
two vertex covers related to \,star$_{\alpha_i}(\alpha_i-\alpha_{i-1})$, $\{v_{\alpha_1}\,, \ldots,
 v_{\alpha_{i-1}}, v_{\alpha_{i-1}+1}, \ldots,$ $v_{\alpha_{i+1}-1},$ $v_{\alpha_{i+1}}\,,$ $\ldots, v_{\alpha_m}\}$, $\{v_{\alpha_1}\,, \ldots, v_{\alpha_m}\}$, for some $i \neq 1$\,, as a consequence of Lemma \ref{Lema2}\,.\\[.5mm]
For the last ones, let
$\mathcal{P}_{\overline{\alpha_i}}=(X_{\alpha_1}\,, \ldots,
X_{\alpha_{i-1}}, X_{\alpha_{i-1}+1}, \ldots, X_{\alpha_{i+1}-1},
X_{\alpha_{i+1}}\,,\ldots,$ $X_{\alpha_m})$,
$\mathcal{P}=(X_{\alpha_1}\,, \ldots, X_{\alpha_m})$, respectively,
be the associated minimal prime ideals but $\mathcal{P} \supseteq
\mathcal{P}_{\alpha_j}$, $j \neq i$. So $I(\mathcal{H})=
\mathcal{P}_{\alpha_1} \cap \ldots \cap \mathcal{P}_{\alpha_{i-1}}
\cap \mathcal{P}_{\overline{\alpha_i}} \cap
\mathcal{P}_{\alpha_{i+1}} \cap \ldots \cap
\mathcal{P}_{\alpha_m}$\,. On the other hand, it is clear what
$\mathcal{P}_{\overline{\alpha_1}}$ denotes. Hence the thesis
follows.
\end{proof}
\begin{remark}\rm{
\noindent Proposition \ref{Km+1star} can be generalized by considering two or more star graphs with centers as many vertices of $K_m$\,. It is sufficient to iterate that procedure for each pair of vertex covers related to the star graphs which are present.}
\end{remark}
Finally, let's consider the structure of the ideal of vertex covers of $I(\mathcal{H})$ when all the vertices of $K_m$ have degree at least $m$.
\begin{theorem} \label{Km+allstar}
Let $\mathcal{H}$ be the connected graph with $n$ vertices $v_1,\ldots, v_n$ formed by the union of a complete graph $K_m, m<n$, with vertices $v_{\alpha_1}, \ldots, v_{\alpha_m},$ $1 < \alpha_1 < \alpha_2 < \ldots < \alpha_m=n$, and of $m$ star graphs with centers each of the vertices of $K_m$. The ideal of vertex covers of $I(\mathcal{H})$ has $m+1$ monomial squarefree generators and it is\\
$I_c(\mathcal{H})=(X_{\alpha_1}X_{\alpha_2}\!\cdot\!\cdot\!\cdot X_{\alpha_m},
X_{1}\!\cdot\!\cdot\!\cdot X_{\alpha_1-1}X_{\alpha_2}\!\cdot\!\cdot\!\cdot X_{\alpha_m}, X_{\alpha_1}X_{\alpha_1+1}\!\cdot\!\cdot\!\cdot X_{\alpha_2-1}X_{\alpha_3}\!\cdot\!\cdot\!\cdot X_{\alpha_m},
\ldots,
X_{\alpha_1}\!\cdot\!\cdot\!\cdot X_{\alpha_{m-2}} X_{\alpha_{m-2}+1}\!\cdot\!\cdot\!\cdot X_{\alpha_{m-1}-1}X_{\alpha_m},
X_{\alpha_1}\!\cdot\!\cdot\!\cdot X_{\alpha_{m-1}}X_{\alpha_{m-1}+1}\!\cdot\!\cdot\!\cdot X_{\alpha_m-1})\,.$
\end{theorem}
\begin{proof}
A minimal vertex cover of $\mathcal{H}$ must be $\{v_{\alpha_1}\,, \ldots, v_{\alpha_m}\}$ and there cannot exist minimal vertex covers with a smaller number of vertices. So $\alpha_0 (\mathcal{H})=m$.\\
Other minimal vertex covers of $\mathcal{H}$ are those related to star graphs on $\alpha_1$ and $\alpha_i-\alpha_{i-1}$ vertices, star$_{\alpha_1}(\alpha_1)$ and star$_{\alpha_i}(\alpha_i-\alpha_{i-1})$ respectively, for any $i=2,\ldots,m$\,, in which the centers are missing. With the notations of Proposition \ref{Km+1star}, let $\mathcal{P}=(X_{\alpha_1}\,, \ldots, X_{\alpha_m})$, $\mathcal{P}_{\overline{\alpha_1}}=(X_1, \ldots, X_{\alpha_1-1}\,, X_{\alpha_2}, \ldots, X_{\alpha_m})$, $\mathcal{P}_{\overline{\alpha_i}}=(X_{\alpha_1}\,, \ldots,$ $X_{\alpha_{i-1}}, X_{\alpha_{i-1}+1}, \ldots, X_{\alpha_i-1},$ $ X_{\alpha_{i+1}}\,,\ldots,$ $X_{\alpha_m})$, for any $i=2, \ldots, m$\,, be the associated minimal prime ideals. Then
$I(\mathcal{H})= \mathcal{P} \cap \bigg(\,\displaystyle{\bigcap_{i=1}^m \mathcal{P}_{\overline{\alpha_i}}}\bigg)$\,, and ht $I(\mathcal{H})=$ $m$\,. Thus $I_c(\mathcal{H})=(X_{\alpha_1}\cdots X_{\alpha_m}, X_1\cdots X_{\alpha_1-1} X_{\alpha_2}\cdots X_{\alpha_m}, \ldots,$ $X_{\alpha_1}\cdots X_{\alpha_{m-1}}X_{\alpha_{m-1}+1}\cdots X_{\alpha_m-1})$, and the thesis follows.
\end{proof}
\begin{example}\rm{
Let $\mathcal{H}$ be the connected graph with
$V(\mathcal{H})=\{v_1,\ldots,v_4,\ldots,v_{11}\}$ given by $K_4 $
$\cup$\, star$_1(2)$ $ \cup$\, star$_2(4)$ $ \cup$\, star$_3(2)$ $
\cup$\, star$_4(3)$\,.
The ideal of vertex covers is\\[1mm]
$I_c(\mathcal{H})=(X_1X_2X_3X_4, X_2X_3X_4X_5, X_1X_3X_4X_6X_7X_8, X_1X_2X_4X_9, X_1X_2X_3X_{10}X_{11})$.
}
\end{example}

Let's now analyze the structure of the ideals of vertex covers for the class of non-squarefree edge ideals associated to the connected graphs on $n$ vertices\, $\mathcal{K}'=\mathcal{H} \cup \{v_i,v_i\}$, for some $i=1,\ldots,n$\,.
Note that the class $\mathcal{K}'$ is larger than $\mathcal{K}$ because $\mathcal{K}'$ may have loops on vertices that don't belong to $K_m$\,.\\[.5mm]

First let's enunciate two lemmas that give the generators of the ideals of vertex covers for the edge ideals of a complete graph with loops and a star graph with loops, respectively.\\ Their proofs are a direct consequence of Property \ref{vertcov}\,.
\begin{lemma} \label{Lema3}
Let $K'_m$ be the complete graph with loops having vertices \, $v_1,\ldots, v_m$\,. The ideal of vertex covers of $\,I(K'_m)$ is generated at most by $m-1$ monomials. In particular,\\
(a) if there are loops in all the vertices,\, $I_c(K'_m)=(X_1X_2\cdots X_m),$\\
(b) if there are loops in \,$r<m$ vertices, $v_{t_1},\ldots,
v_{t_r}$\,, \,$\{t_1,\ldots,t_r\} \subseteq \{1, \ldots, m\},$\par
$I_c(K'_m)$ has \,$m\!-\!r$ generators and it is \par
$\big(\{X_{\sigma_{1}}\!\cdot\cdot\cdot\!X_{\sigma_{m-1}}\ \!|\! \
\sigma_j\!=\!t_j,\, \forall \,j\!=\!1,\ldots,r; \sigma_i\in
\{1,\ldots,m\}\!\setminus \{t_1,\ldots,t_r\}, \forall
\,i\!\neq\! j\}\big).$
\end{lemma}
\begin{lemma} \label{Lema4}
Let star$'_n(n)$ be the star graph with loops having vertices \, $v_1,\ldots, v_n$\,. The ideal of vertex covers of $\,I($star$'_n(n))$ has at most 2 \mbox{generators}. In particular,\\
(a) if the loops are in the vertices $v_1,\ldots, v_{n-1}$\,, $I_c\,($star$'_n(n))=(X_1\cdots X_{n-1}),$\\
(b) if the loops are in \,$v_3, v_{n-2}$\,,\, $I_c\,($star$'_n(n))=(X_1\cdots X_{n-1}, \,X_3X_{n-2}X_n),$\\
(c) if there are loops in the center and in the vertices\,
$v_{t_1},\ldots, v_{t_s}$\,, \,$\{t_1,\ldots,t_s\}$\par $\subseteq
\{1, \ldots, n\!-\!1\}$\,, \, $I_c\,($star$'_n(n))=(X_{t_1}\cdots
X_{t_s}X_n)\,.$
\end{lemma}
\vspace{1mm}
The structure of the ideal of vertex covers of $I(\mathcal{K}')$ is well described by the following
\begin{theorem} \label{Km+star+loopsout}
Let $\mathcal{K}'$ be the connected graph with $n$ vertices $v_1,\ldots, v_n$ formed by the union of\,: (i) the complete graph $K_m\,, m<n$, with vertices $v_{\alpha_1}, \ldots, v_{\alpha_m},$ $1 \leqslant \alpha_1 < \alpha_2 < \ldots < \alpha_m=n$\,; (ii) star graphs \mbox{star$_{\alpha_i}(\alpha_i-\alpha_{i-1})$} on vertices $v_{\alpha_{i-1}+1},\ldots,v_{\alpha_i}\,, \forall \,i=1,\ldots,m$\,, and the index $\alpha_0$ indicates $0$\,; (iii) loops in some vertices, say $v_{\alpha_2}, v_{\alpha_4}, v_{\alpha_5}, v_{\alpha_{m-3}}, v_{\alpha_{m-1}},$ $v_{\alpha_{i-1}+j_1}, \ldots, v_{\alpha_{i-1}+j_{p_i}}$\,, where 
$\{j_1,$ $\dots,j_{p_i} \} \subseteq \{1, \ldots, \alpha_i\!-\!\alpha_{i-1}\!-\!1\}$\,.
The ideal of vertex covers of $I(\mathcal{K}')$ has at most $m+1$ monomial generators and it is\\[1mm]
$I_c(\mathcal{K}')=(X_{j_1}\!\cdot\!\cdot\!\cdot
X_{j_{p_1}}X_{\alpha_1}X_{\alpha_1+j_1} \!\cdot\!\cdot\!\cdot
X_{\alpha_1+j_{p_2}}X_{\alpha_2} X_{\alpha_2+j_1}
\!\cdot\!\cdot\!\cdot X_{\alpha_{m-1}+j_{p_m}}X_{\alpha_m},$\\ $
X_{1}\!\cdot\!\cdot\!\cdot X_{\alpha_{1}-1}X_{\alpha_1+j_1}
\!\cdot\!\cdot\!\cdot
X_{\alpha_1+j_{p_2}}X_{\alpha_2}X_{\alpha_2+j_1}
\!\cdot\!\cdot\!\cdot X_{\alpha_{m-1}+j_{p_m}}X_{\alpha_m},$\\ $
X_{j_1}\!\cdot\!\cdot\!\cdot X_{j_{p_1}}X_{\alpha_1}X_{\alpha_1+1}
\!\cdot\!\cdot\!\cdot X_{\alpha_2-1}X_{\alpha_2}X_{\alpha_2+j_1}
\!\cdot\!\cdot\!\cdot
X_{\alpha_{m-1}+j_{p_m}}X_{\alpha_m},\ldots,$\\ $
X_{j_1}\!\cdot\!\cdot\!\cdot
X_{\alpha_{m-3}+j_{p_{m-2}}}X_{\alpha_{m-2}}
X_{\alpha_{m-2}+1}\!\cdot\!\cdot\!\cdot
X_{\alpha_{m-1}-1}X_{\alpha_{m-1}}
X_{\alpha_{m-1}+j_1}\!\cdot\!\cdot\cdot$\\ $
X_{\alpha_{m-1}+j_{p_m}}X_{\alpha_m},$ $
X_{j_1}\!\cdot\!\cdot\!\cdot
X_{\alpha_{m-2}+j_{p_{m-1}}}X_{\alpha_{m-1}}
X_{\alpha_{m-1}+1}\!\cdot\!\cdot\!\cdot X_{\alpha_m-1})\,.$
\end{theorem}
\begin{proof}
The correspondence between minimal vertex covers of the
graph $\mathcal{H}$ and minimal prime ideals of the edge ideal of
$\mathcal{H}$ extends to the graph $\mathcal{K}'$. So the number of
minimal prime ideals $\wp$ of $I(\mathcal{K}')$ is the same than
that of minimal primes of $I(\mathcal{H})$. Moreover, by Property
\ref{vertcov}, the generators of the prime ideals $\wp$ are those
related to $\mathcal{H}$ multiplied by monomials $X_i$
 whenever $X_i \notin \wp\,,\, i=1,\ldots,n$\,. The latter monomials represent the vertices with loops of $\mathcal{K}'$ for which the
corresponding variables are missing in $\wp$\,. The assertion holds through some
computation taking in consideration Theorem \ref{Km+allstar}\,.
\end{proof}
Significative particular cases concern the graphs of the class such that:
(a) some star graphs are missing,
(b) the loops lie only on the vertices of $K_m$, and
(c) the loops lie only on the vertices not belonging to $K_m$.
\begin{corollary} Let $\mathcal{K}'$ be as in Theorem \ref{Km+star+loopsout}, but there are in it less than $m$ star graphs star$_{\alpha_i}(\alpha_i-\alpha_{i-1})$,
suppose for \,$i=2,3,6,$ $m\!-\!3,$ $m\!-\!2,m$\,. The ideal of vertex covers of $I(\mathcal{K}')$
has at most $m$ monomial generators and it is\\
$I_c(\mathcal{K}')=(X_{\alpha_1+j_1} \!\cdot\!\cdot\!\cdot
X_{\alpha_1+j_{p_2}}X_{\alpha_2}X_{\alpha_2+j_1}
\!\cdot\!\cdot\!\cdot X_{\alpha_2+j_{p_3}} X_{\alpha_3}X_{\alpha_4}
\!\cdot\!\cdot\!\cdot X_{\alpha_{m-1}+j_{p_m}}X_{\alpha_m},$\\ $
X_{\alpha_1}X_{\alpha_1+1} \!\cdot\!\cdot\!\cdot
X_{\alpha_2-1}X_{\alpha_2}X_{\alpha_2+j_1} \!\cdot\!\cdot\!\cdot
X_{\alpha_2+j_{p_3}} X_{\alpha_3}X_{\alpha_4} \!\cdot\!\cdot\!\cdot
X_{\alpha_{m-1}+j_{p_m}}X_{\alpha_m},\ldots,$\\ $
X_{\alpha_1}\!\cdot\!\cdot\!\cdot
X_{\alpha_{m-3}+j_{p_{m-2}}}X_{\alpha_{m-2}} X_{\alpha_{m-1}}
X_{\alpha_{m-1}+j_1}\!\cdot\!\cdot\cdot$ $
X_{\alpha_{m-1}+j_{p_m}}X_{\alpha_m},$\\ $
X_{\alpha_1}\!\cdot\!\cdot\!\cdot
X_{\alpha_{m-3}+j_{p_{m-2}}}X_{\alpha_{m-2}}X_{\alpha_{m-1}}
X_{\alpha_{m-1}+1}\!\cdot\!\cdot\!\cdot X_{\alpha_m-1})\,.$
\end{corollary}
\begin{corollary} Let $\mathcal{K}$ be the subgraph of $\mathcal{K}'$ having loops only in the vertices of $K_m$\,.
 The ideal of vertex covers of $I(\mathcal{K})$ has at most $m+1$ monomial generators.
 For instance, if the loops lie on $v_{\alpha_2}, v_{\alpha_4}, v_{\alpha_5}, v_{\alpha_{m-3}}, v_{\alpha_{m-1}}$\,, it is:\\
$I_c(\mathcal{K})=(X_{\alpha_1}\!\cdot\!\cdot\!\cdot X_{\alpha_m},
X_{1}\!\cdot\!\cdot\!\cdot
X_{\alpha_1-1}X_{\alpha_2}\!\cdot\!\cdot\!\cdot X_{\alpha_m},
X_{\alpha_1}X_{\alpha_1+1}\!\cdot\!\cdot\!\cdot
X_{\alpha_2-1}X_{\alpha_2}X_{\alpha_3}\!\cdot\!\cdot\!\cdot
X_{\alpha_m},$ $\ldots,$ $ X_{\alpha_1}\!\cdot\!\cdot\!\cdot
X_{\alpha_{m-2}} X_{\alpha_{m-2}+1}\!\cdot\!\cdot\!\cdot
X_{\alpha_{m-1}-1}X_{\alpha_{m-1}} X_{\alpha_m},$
$X_{\alpha_1}\!\cdot\!\cdot\!\cdot
X_{\alpha_{m-1}}X_{\alpha_{m-1}+1} \!\cdot\!\cdot\!\cdot
X_{\alpha_m-1})\,.$
\end{corollary}
\begin{corollary}
Let \,$\mathcal{H} \cup \{v_h,v_h\}$ be the connected graph on $n$
vertices $v_1,\ldots,$ $v_m,\ldots, v_n\,, \,\forall \,h=m+1,\ldots,n$\,.
The ideal of vertex covers for the edge ideal of $\mathcal{H} \cup \{v_h,v_h\}$ is generated by the $m$ monomials\,
$X_2X_3\cdots X_mX_{m+1}\cdots X_n\,,$ $X_1X_3\cdots X_mX_{m+1}\cdots X_n\,,\,\ldots, X_1X_2\cdots X_{m-1}X_{m+1}\cdots X_n$\,.
\end{corollary}

\begin{example}\rm{
Let $\mathcal{K}'$ be the connected graph with
$V(\mathcal{K}')=\{v_1,v_2,v_3,\ldots,v_{11}\}$ given by $K_3 $
$\cup$\, star$_1(4)$ $ \cup$\, star$_2(4)$ $ \cup$\, star$_3(3)$ $
\cup$ $ \{v_3,v_3\}$ $ \cup$ $ \{v_5,v_5\}$ $ \cup$ $ \{v_7,v_7\}$ $
\cup $ $\{v_9,v_9\}$\,. The ideal of vertex covers is\\[1mm]
$I_c(\mathcal{K}')= (X_1X_2X_3X_5X_7X_9, \,
X_2X_3X_4X_5X_6X_7X_9, \, X_1X_3X_5X_7X_8X_9).$}
\end{example}


\section{Ideals of vertex covers of linear type}

\hspace{6mm} Let $R$ be a noetherian ring  and let $I=(f_1,\ldots,f_s)$ be an ideal of $R$.\\
The Rees algebra $\Re(I)$ of $I$ is defined to be the $R$-graded
algebra $\bigoplus_{i\geqslant 0}I^i$. It can be identified with the
$R$-subalgebra of $R[t]$ generated by $It$, where $t$ is an
indeterminate on $R$. Let we consider the epimorphism of graded
$R$-algebras $\varphi:R[T_1,\ldots,T_s] \rightarrow
\Re(I)=R[f_1t,\ldots,f_st]$ defined by $\varphi(T_i)=f_it$, $i=1,\ldots,s$. \\
The ideal $N=\ker\varphi$ of $R[T_1,\ldots,T_s]$ is  $R$-homogeneous
and we denote $N_i$ the $R$-homogeneous component of degree $i$ of
$N$. The elements of $N_1$ are called linear relations. If
$A=(a_{ij})$, $i=1,\ldots,r$, $j=1,\ldots,s$ is the relation matrix
of $I$, then $g_i=\sum_{j=1}^s a_{ij}T_j$, $i=1,\ldots,r$, belong to
$N$ and $R[T_1,\ldots,T_s]/J$, with $J=(g_1,\ldots,g_r)$, is
isomorphic to the symmetric algebra $Sym_R(I)$ of $I$. The
generators $g_i$ of $J$ are all linear in the
variables $T_j$. \\
The natural map $\psi: Sym_R(I)\rightarrow \Re(I)$ is a surjective
homomorphism of $R$-algebras. $I$ is called of \textit{linear type}
if $\psi$
is an isomorphism, that is $N=J$.\\
Now, let $K$ be a field, $R=K[X_1,\ldots,X_n]$ be the polynomial
ring, $I \subset R$ be  a graded ideal whose generators
$f_1,\ldots,f_s$ are all of the same degree.
Let $S=R[T_1,\ldots,T_s]$ be the polynomial ring over $R$ in the
variables $T_1,\ldots,T_s$. Then we define a bigrading of $S$ by
setting $\deg(X_i)=(1,0)$ for $i=1,\ldots,n$ and $\deg(T_j)=(0,1)$ for
$j=1,\ldots,s$.
 Consider the presentation $
\varphi:R[T_1,\ldots,T_s] \rightarrow \Re(I)$, $ \varphi(T_i)=f_it,
\quad i=1,\ldots,s.$ If  $I=(f_1,\ldots,f_s) \subset R$ is a
monomial ideal, for all $1 \leqslant i<j \leqslant s$ we set $
f_{ij}=\frac{f_i}{GCD(f_i,f_j)}$ and $ g_{ij}=f_{ij}T_j -f_{ji}T_i,$
then $J$ is generated by
$\{g_{ij}\}_{1 \leqslant i<j \leqslant s}$.\\
Our aim is to investigate classes of monomial ideals arising from
graphs for which the linear relations $g_{ij}$ form a system of
generators for $N$, i.e. $N=J$. \\[1mm] Consider the ideals of vertex
covers $I_c(\mathcal{H})$ and $I_c(\mathcal{K}')$\,.

\begin{theorem}\label{3.1}
Let $R=K[X_1, \ldots,X_m, \ldots,  X_n]$. $I_c(\mathcal{H})$ is of linear type.
\end{theorem}
\begin{proof} Let $f_1,\ldots,f_{m+1}$ be the minimal system of
monomial generators of $I_c(\mathcal{H})$, where
$f_1=X_{\alpha_1}X_{\alpha_2}\!\cdot\!\cdot\!\cdot X_{\alpha_m}$,
 $f_2=X_{1}\!\cdot\!\cdot\!\cdot
X_{\alpha_1-1}X_{\alpha_2}\!\cdot\!\cdot\!\cdot X_{\alpha_m},$
$\ldots, $
 $f_m=X_{\alpha_1}\!\cdot\!\cdot\!\cdot X_{\alpha_{m-2}}
X_{\alpha_{m-2}+1}\!\cdot\!\cdot\!\cdot
X_{\alpha_{m-1}-1}X_{\alpha_m}$, $f_{m+1}=
X_{\alpha_1}\!\cdot\!\cdot\!\cdot
X_{\alpha_{m-1}}X_{\alpha_{m-1}+1}\!\cdot\!\cdot\!\cdot
X_{\alpha_m-1}$ (Theorem \ref{Km+allstar}).\\
We prove that the linear relations $g_{ij}=f_{ij}T_j-f_{ji}T_i$ form
a Gr\"{o}bner basis of $N$ with respect to a monomial order $\prec$
on the polynomial ring $R[T_1,\ldots,T_{m+1}]$. Denote by $H$ the
ideal $(f_{ij}T_j:1\leqslant i<j\leqslant m+1)$. To show that $g_{ij}$ are a
Gr\"{o}bner basis of $N$ we suppose that the claim is false. Since
the binomial relations are known to be a Gr\"{o}bner basis of $N$,
there exists a binomial
$\underline{X}^a\underline{T}^{\alpha}-\underline{X}^b\underline{T}^{\beta}
\in N$, where $\underline{X}^a=X_1^{a_1}\cdots X_n^{a_n}$,
$\underline{X}^b=X_1^{b_1}\cdots X_n^{b_n}$,
$\underline{T}^{\alpha}=T_1^{\alpha_1}\cdots
T_{m+1}^{\alpha_{m+1}}$, $\underline{T}^{\beta}=T_1^{\beta_1}\cdots
T_{m+1}^{\beta_{m+1}}$, and the initial monomial of
$\underline{X}^a\underline{T}^{\alpha}-\underline{X}^b\underline{T}^{\beta}$
is not in $H$. More precisely, we assume that $T^{\alpha},T^{\beta}
$ have no common factors and that both
$\underline{X}^a\underline{T}^{\alpha}$ and
$\underline{X}^b\underline{T}^{\beta}$ are not in $H$.\\
Let $i$ be the smallest index such that $T_i$ appears in
$\underline{T}^{\alpha}$ or in $\underline{T}^{\beta}$. Since
$\underline{X}^a\underline{T}^{\alpha}-\underline{X}^b\underline{T}^{\beta}\in
N$, then $f_i$ divides
$\underline{X}^b\varphi(\underline{T}^{\beta})$, where
$\varphi(T_i)=f_it$. If $f_i|\underline{X}^b$, then let $T_j$ be any
of the variables of $\underline{T}^{\beta}$. One has $f_{ij}T_j|
f_iT_j|\underline{X}^b\underline{T}^{\beta}$ for $i<j$. This is a
contradiction by assumption (because
$\underline{X}^b\underline{T}^{\beta} \notin H$). \\
Hence $f_i\nmid \underline{X}^b$. Let $X_{i_1} \prec \ldots \prec
X_{i_{s}}$ be a total term order on the variables of $f_i$, and let
$f_i=X_{i_1}\cdots X_{i_{s}}$. Let $i_{k_1}, \ldots, i_{k_t} \in
\{i_1,\ldots,i_{s}\}$ be the indices  such that $X_{i_{k_1}},
\ldots, X_{i_{k_t}}$ don't divide $\underline{X}^b$ and $i_{k_1}$ be
the minimum of the indices such that $X_{i_{k_1}}$ does not divide
$\underline{X}^b$. Then $g_i=f_i/X_{i_{k_1}} \cdots X_{i_{k_t}}$
divides $\underline{X}^b$. Since $X_{i_{k_1}}$ divides
$\underline{X}^b\varphi(\underline{T}^{\beta})$ (because
$f_i|\underline{X}^b\varphi(\underline{T}^{\beta})$), then there
exists $j$ such that $T_j$ appears in $\underline{T}^{\beta}$ and
$X_{i_{k_1}}| f_j$. By the structure of the generators
$f_1,\ldots,f_{m+1}$ of $I_c(\mathcal{H})$ (see Theorem
\ref{Km+allstar})  if $X_{i_{k_1}}| f_i$ and $X_{i_{k_1}}| f_j$ with
$j$ such that $T_j$ is in $\underline{T}^{\beta}$, then
$f_{ij}|g_i$. Hence $f_{ij}$ divides $\underline{X}^b$ and, as a
consequence, $f_{ij}T_j$ divides
$\underline{X}^b\underline{T}^{\beta}$, that is a contradiction
(because $\underline{X}^b\underline{T}^{\beta} \notin H$). It
follows that $N=(g_{ij}: 1\leqslant i < j \leqslant m+1)$, hence
$I_c(\mathcal{H})$ is of linear type.
\end{proof}

\begin{theorem}
Let $R=K[X_1, \ldots,X_m, \ldots,  X_n]$.
$I_c(\mathcal{K}')$ is of linear type.
\end{theorem}
\begin{proof} Let $f_1,\ldots,f_{\ell}$ for $\ell \leqslant m+1$ be the minimal system of
monomial generators of $I_c(\mathcal{K}')$ described in Theorem \ref{Km+star+loopsout}.\\
We prove that the linear relations $g_{ij}=f_{ij}T_j-f_{ji}T_i$,
$1\leqslant i<j\leqslant \ell$,  form a Gr\"{o}bner basis of $N$.  We suppose
that the claim is false. Using the same notation of Theorem
\ref{3.1} there exists a binomial
$\underline{X}^a\underline{T}^{\alpha}-\underline{X}^b\underline{T}^{\beta}
\in N$
  and
$\textrm{in}_{\prec}(
\underline{X}^a\underline{T}^{\alpha}-\underline{X}^b\underline{T}^{\beta})$
is not in $H=(f_{ij}T_j: 1\leqslant i<j\leqslant \ell, \ \ \ell \leqslant m+1)$.\\
Let $i$ be the smallest index such that $T_i$ appears in
$\underline{T}^{\alpha}$ or in $\underline{T}^{\beta}$. Since
$\underline{X}^a\underline{T}^{\alpha}-\underline{X}^b\underline{T}^{\beta}\in
N$, then $f_i$ divides
$\underline{X}^b\varphi(\underline{T}^{\beta})$, where
$\varphi(T_i)=f_it$. If $f_i|\underline{X}^b$, then  $f_{ij}T_j|
f_iT_j|\underline{X}^b\underline{T}^{\beta}$, $i<j$ and  $T_j$ any
of the variables of $\underline{T}^{\beta}$. This is a
contradiction. \\
Hence $f_i\nmid \underline{X}^b$. Let  $f_i=X_{i_1} \dots
X_{i_{s}}$, $i_{k_1}, \ldots, i_{k_t} \in \{i_1,\ldots,i_{s}\}$ be
the indices such that $X_{i_{k_1}}, \ldots, X_{i_{k_t}}$ don't
divide $\underline{X}^b$ and $i_{k_1}$ be the minimum of the indices
such that $X_{i_{k_1}}$ does not divide $\underline{X}^b$. Set
$g_i=f_i/X_{i_{k_1}}\cdots X_{i_{k_t}}$. Since $X_{i_{k_1}}$ divides
$\underline{X}^b\varphi(\underline{T}^{\beta})$, then there exists
$j$ such that $T_j$ appears in $\underline{T}^{\beta}$ and
$X_{i_{k_1}}| f_j$. By the structure of the monomials
$f_1,\ldots,f_{\ell}$ (Theorem \ref{Km+star+loopsout}) if
$X_{i_{k_1}}| f_i$ and $X_{i_{k_1}}| f_j$ with $j$ such that $T_j$
is in $\underline{T}^{\beta}$, then $f_{ij}|g_i$. Hence
$f_{ij}|\underline{X}^b$ and, as a consequence, $f_{ij}T_j |
\underline{X}^b\underline{T}^{\beta}$, that is a contradiction.
Hence $N=(g_{ij}: 1\leqslant i < j \leqslant \ell, \ \ \ell \leqslant  m+1)$. The
thesis follows.
\end{proof}



\end{document}